\documentclass{amsart}
\usepackage{amsmath,amsthm,amssymb,cite}
\usepackage{mathtools}
\usepackage{verbatim}
\mathtoolsset{showonlyrefs}
\usepackage{enumitem}

\newtheorem{theorem}{Theorem}
\newtheorem{lemma}[theorem]{Lemma}
\newtheorem{proposition}[theorem]{Proposition}

\newcommand{\os}[2]{\mathop{\langle {#2} \rangle}\nolimits_{#1}} %LYZ operator
\newcommand{\opp}[1]{\operatorname{\Pi}_{#1}^*} %L^p polar projection body
\newcommand{\app}[1]{\operatorname{\Pi}^{\ast}_{#1,\scriptscriptstyle+}} %asymmetric L^p polar projection body +
\newcommand{\appm}[1]{\operatorname{\Pi}^{\ast}_{#1,\scriptscriptstyle\pm}} %asymmetric L^p polar projection body \pm

\newcommand{\ofpp}[1]{\operatorname{\Pi}^{\ast}_{#1}\!} % L^p polar projection body
\newcommand{\fpp}[2]{\operatorname{\Pi}^{\ast,{\it #1}}_{#2}\!} %fractional L^p polar projection body
\newcommand{\ppp}[2]{\operatorname{\Pi}^{\ast,{\it #1}}_{#2,\scriptscriptstyle+}}
\newcommand{\ppm}[2]{\operatorname{\Pi}^{\ast,{\it #1}}_{#2,\scriptscriptstyle-}}
\newcommand{\pppm}[2]{\operatorname{\Pi}^{\ast,{\it #1}}_{#2,\scriptscriptstyle\pm}}
\newcommand{\om}{\operatorname{\rm Z}_p} 
\newcommand{\omp}{\operatorname{\rm Z}_p^*\!} 

\newcommand{\sln}{\operatorname{SL}(n)}
\newcommand{\gln}{\operatorname{GL}(n)}

\newcommand {\R} {\mathbb R}
\newcommand{\B}{{B^n}}
\newcommand {\sn} {{\mathbb S}^{n-1}}

\renewcommand{\d}{\,\mathrm{d}}

\newcommand{\vol}[1]{ \vert #1\vert}
\newcommand{\plus}[2]{(#1)_{\scriptscriptstyle +}^{ #2}}
\newcommand{\pl}[1]{#1_{\scriptscriptstyle +}}
\newcommand{\minus}[2]{(#1)_{\scriptscriptstyle -}^{ #2}}
\newcommand{\plusminus}[2]{(#1)_{\scriptscriptstyle \pm}^{ #2}}
\newcommand{\mn}[1]{#1_{\scriptscriptstyle -}}
\newcommand{\plmn}[1]{#1_{\scriptscriptstyle \pm}}
\renewcommand \div {\operatorname{div}}

\renewcommand{\chi}{\operatorname{1}}
\newcommand{\ra}[1]{\mathbin{\tilde +_{\scriptscriptstyle#1}}}
\newcommand{\Bp}{B_{p'\!,\scriptscriptstyle+}}

\title{Affine fractional $L^p$ Sobolev inequalities}

\author{Juli\'an Haddad}
\address{Departamento de Matem\'atica, ICEx, Universidade Federal de Minas Gerais, 30.123-970, Belo Horizonte, Brazil}
\email{jhaddad@mat.ufmg.br}

\author{Monika Ludwig}
\address{Institut f\"ur Diskrete Mathematik und Geometrie,
Technische Universit\"at Wien,
Wiedner Hauptstra\ss e 8-10/1046,
1040 Wien, Austria}
\email{monika.ludwig@tuwien.ac.at}

\begin{document}

\maketitle
\begin{abstract}
Sharp affine fractional $L^p$ Sobolev inequalities for functions on $\R^n$ are
established. The new inequalities are stronger than (and directly imply) the sharp fractional $L^p$ Sobolev inequalities. They are fractional versions of the affine $L^p$ Sobolev inequalities of Lutwak, Yang, and Zhang. In addition, affine fractional asymmetric $L^p$ Sobolev inequalities are established.

\end{abstract}

\section{Introduction}
Sharp fractional $L^2$ Sobolev inequalities are receiving  increasing attention in the last decades. 
They are central in the study of solutions of equations involving the fractional Laplace operator $(-\Delta)^{1/2}$ which arises naturally in many non-local problems such as the stationary form of reaction-diffusion equations \cite{cafarelli2012}, the Signorini problem (and its equivalent formulation as the thin obstacle problem) \cite{cafarelli2006}, and the Dirichlet-to-Neumann operator of harmonic functions in the half space \cite{silvestre2007}.
Also, the general operators $(-\Delta)^s$ for $s \in (0,1)$ arise in stochastic theory, associated with symmetric Levy processes (see \cite{silvestre2007} and the references therein).

Let $0< s<1$ and $1\le p < n/s$.  The fractional $L^p$ Sobolev inequalities state that
\begin{align}\label{eq_fracsob}
\Vert f\Vert_{{\frac {n p}
{n-p s}}}^p
\le
\sigma_{n,p,s}\int_{\R^n} \int_{\R^n} \frac{\vert f(x)-f(y)\vert^p}{\vert x-y \vert^{n+p s}}\d x\d y
\end{align}
for $f\in W^{s,p}(\R^n)$, the fractional $L^p$ Sobolev space of  functions $f\in L^p(\R^n)$ with finite right side in \eqref{eq_fracsob} (see, for example, \cite{Mazya}). 
In general, the optimal constants  $\sigma_{n,p,s}$ and extremal functions are not known (see \cite{BMS} for a conjecture).
Equality is always attained in \eqref{eq_fracsob}. 
For $p=1$, the extremal functions of \eqref{eq_fracsob}  are  multiples of indicator functions of balls and the constants are explicitly known. 
The only further known case is $p=2$, where the constants  $\sigma_{n,2,s}$ can be obtained by duality from Lieb's sharp Hardy--Littlewood--Sobolev inequalities \cite{Lieb83} (see, for example,  \cite{Carlen2017}). 
The asymptotic behavior of $\sigma_{n,p,s}$ as $s \to 1^-$ was studied in \cite{BBM}. 
Almgren and Lieb \cite{AlmgrenLieb} and Frank and Seiringer \cite{FrankSeiringer} showed that the extremal functions of \eqref{eq_fracsob} are  radially symmetric and   of constant sign. 

By a result of Bourgain, Brezis, and Mironescu~\cite{BourgainBrezisMironescu}, 
\begin{equation}\label{eq_BBM}
\lim_{s\to 1^-} p (1-s) \int_{\R^n} \int_{\R^n} \frac{\vert f(x)-f(y)\vert^p}{\vert x-y \vert^{n+p s}}\d x\d y = \alpha_{n,p} \int_{\R^n} \vert \nabla f(x)\vert^p \d x
\end{equation}
for $f\in W^{1,p}(\R^n)$,  the Sobolev space of $L^p$ functions $f$ with weak $L^p$ gradient $\nabla f$, %where $\alpha_{n,p}>0$ is an explicitly known constant. 
where 
\begin{equation}\label{eq_alpha}
\alpha_{n,p}= \int_{\sn} \vert\langle \xi,\eta\rangle \vert^p \d\xi\
\end{equation}
for any $\eta\in\sn$. Here, integration on the unit sphere $\sn$ is with respect to the $(n-1)$-dimensional Hausdorff measure, $\omega_n$ is the volume of the $n$-dimensional unit ball and $\langle \cdot,\cdot\rangle$ is the inner product on $\R^n$. 
For $p=1$ and $p=2$, this allows to deduce the sharp $L^p$ Sobolev inequalities from \eqref{eq_fracsob} by calculating the limit of $\sigma_{n,p,s}/(1-s)$ as $s\to 1^-$. 

Zhang \cite{Zhang99} and Lutwak, Yang, and Zhang  \cite{LYZ2002b} obtained the following sharp affine $L^p$ Sobolev inequality that is significantly stronger than the classical $L^p$ Sobolev inequality:
\begin{align}
\label{eq_affineLpSobolev}
\Vert f\Vert_{{\frac {n p}{n-p}}}^p
\le
\sigma_{n,p} \frac {n \omega_n^{\frac{n+p}n}} {\alpha_{n,p}}  \vol{\ofpp p\, f}^{-\frac {p }n}
\le
\sigma_{n,p}\int_{\R^n}  \vert \nabla f(x) \vert^p \d x
\end{align}
for $f\in W^{1,p}(\R^n)$ and $1<p<n$, where the inequality between the first and third terms is the classical $L^p$ Sobolev inequality and  the optimal constants $\sigma_{n,p}$ were determined by Aubin \cite{Aubin1976} and Talenti \cite{Talenti1976}. We have rewritten the explicit constant for the first inequality from \cite{LYZ2002b} using  \eqref{eq_alpha}. Here $\ofpp p\, f$ is the $L^p$ polar projection body of $f$, a convex body associated to $f$ that was introduced  with different notation in  \cite{LYZ2002b} (see Section \ref{sec_lp_polar}), and $\vol{\cdot}$ is the $n$-dimensional Lebesgue measure. %The constants $\sigma_{n,p}$ are explicitly known and $\omega_n$ is the volume of the $n$-dimensional unit ball. The first inequality in \eqref{eq_affineLpSobolev} is an affine inequality, which means that it remains unchanged under translations and $\sln$ transformations of $f$. 

The main aim of this paper is to establish affine fractional $L^p$ Sobolev inequalities that are stronger than the Euclidean fractional $L^p$ Sobolev inequalities from \eqref{eq_fracsob} and are fractional counterparts of \eqref{eq_affineLpSobolev}. The case $p=1$ was studied in \cite{HaddadLudwig_fracsob}, so from now on we  let $p>1$.

\begin{theorem}\label{thm_afracsobo}
Let $0<s<1$ and $1<p<n/s$. For $f\in W^{s,p}(\R^n)$, 
\begin{multline}
\Vert f\Vert_{{\frac {n p}{n-p s}}}^p
\\\le
\sigma_{n,p,s} n\omega_n^{\frac{n+p s}n}  \Big(\frac 1n \int_{\sn}\!\!\! \big( \int_0^\infty \!\!t^{p s-1} \!\int_{\R^n}\! \vert f(x+ t\xi)-f(x)\vert^p\d x \d t\big)^{-\frac n{p s}}\d \xi\Big)^{-\frac {p s}n}\\
\le
\sigma_{n,p,s}\int_{\R^n} \int_{\R^n} \frac{\vert f(x)-f(y)\vert^p}{\vert x-y \vert^{n+p s}}\d x\d y.
\end{multline}
There is equality in the first inequality if and only if $f=h_{s,p}\circ \phi$ for some $\phi\in\gln$, where $h_{s,p}$ is an extremal function of \eqref{eq_fracsob}. There is equality in the second inequality if $f$ is radially symmetric.
\end{theorem}

In order to prove Theorem \ref{thm_afracsobo}, we introduce the $s$-fractional $L^p$ polar projection body $\fpp s p f$ associated to $f$, defined as the star-shaped set whose gauge function for $\xi\in\sn$ is 
\begin{equation}\label{eq_defpp}
\Vert{\xi}\Vert_{\fpp {s}p f}^{p s}= \int_0^\infty t^{-p s-1}\, \int_{\R^n}\vert f(x +t\xi)-f(x)\vert^p \d x\d t
\end{equation}
(see Section \ref{sec_fracprojectionbody} for details). The affine fractional Sobolev inequality  now can be written as
\begin{equation}\label{eq_fracsobvol}
\Vert f\Vert_{{\frac {n p}{n-p s}}}^p
\le
\sigma_{n,p,s} n\omega_n^{\frac{n+p s}n} \vol{\fpp {s} pf}^{- \frac{p s}n}.
\end{equation}
Since both sides of \eqref{eq_fracsobvol} are invariant under translations of $f$, and for volume-preserving linear transformations $\phi: \R^n\to \R^n$,
\[ \fpp {s}p \,(f\circ\phi^{-1}) = \phi \fpp {s}p f,\]
it follows that \eqref{eq_fracsobvol} is an affine inequality. In Theorem \ref{thm_limitM}, we will show that
\begin{equation}\label{eq_limitvol}
	\lim_{s\to 1^-} p(1-s) \vol{\fpp s p f}^{-\frac{p s}n} =  \vol{\opp pf}^{-\frac p n},
\end{equation}
which establishes the connection to the $L^p$ polar projection bodies introduced by Lutwak, Yang and Zhang  \cite{LYZ2002b}. 

\goodbreak
In Section \ref{sec_fracaprojectionbody} we introduce fractional asymmetric $L^p$ polar projection bodies as fractional counterparts of the asymmetric $L^p$ polar projection bodies of Haberl and Schuster
\cite{Haberl:Schuster2}, which in turn are functional versions of the asymmetric $L^p$ polar projection bodies of convex bodies introduced in \cite{Ludwig:Minkowski}. We obtain affine fractional asymmetric $L^p$~Sobolev inequalities for non-negative functions that are stronger than the inequalities for the symmetric fractional $L^p$ polar projection bodies. 

In the proofs of the main results, we use anisotropic fractional Sobolev norms, which were introduced in \cite{Ludwig:fracperi, Ludwig:fracnorm} and depend on a star-shaped set $K\subset \R^n$. In Section \ref{sec_optimal} we discuss  which choice of $K$ (with given volume) gives the minimal fractional Sobolev norm and connect it to the corresponding quest for an optimal $L^p$ Sobolev norm solved by Lutwak, Yang, and Zhang \cite{LYZ2006}.

\section{Preliminaries}\label{sec_prelim}

We collect results on function spaces, Schwarz symmetrization, star-shaped sets, anisotropic Sobolev norms and $L^p$ polar projection bodies, that will be used in the following.

\subsection{Function spaces}
For $p\ge 1$ and measurable $f:\R^n\to\R$, let
\[ \Vert f\Vert_p =\Big(\int_{\R^n} \vert f(x)\vert^p\d x \Big)^{1/p}.\]
We set $\{f\ge t\}=\{ x\in \R^n: f(x)\ge t\}$ for $t\in\R$ and use similar notation for level sets, etc. We say that $f$ is non-zero, if $\{f\ne 0\}$ has positive measure, and we identify functions that are equal up to a set of measure zero. For $p\ge 1$, let
	\[L^{p}(\R^n) = \Big\{f: \R^n\to\R: f \text{ is measurable},  \Vert f\Vert_p < \infty \Big\}.\]
Here and below, when we use measurability and related notions, we refer to  the $n$-dimensional Lebesgue measure on $\R^n$.

For $0<s<1$ and $p\ge 1$, we define the fractional  Sobolev space  $W^{s,p}(\R^n)$ as
	\[W^{s,p}(\R^n) = \Big\{f \in L^p(\R^n) :  \int_{\R^n}\int_{\R^n} \frac{|f(x)-f(y)|^p}{|x-y|^{n+p s}} \d x \d y < \infty \Big\}.\]
For $p\ge 1$, we set
	\[W^{1,p}(\R^n) = \big\{f \in L^p(\R^n) :  \vert \nabla f \vert \in L^p(\R^n)\big\},\]
where $\nabla f$ is the weak gradient of $f$.	
	
\subsection{Symmetrization} 

For  a set $E\subset \R^n$, the indicator function $\chi_E$ is defined by $\chi_E(x)=1$ for $x\in E$ and $\chi_E(x)=0$ otherwise.
Let $E \subseteq \R^n$ be a Borel set of finite measure.  
The Schwarz symmetral of $E$, denoted by $E^\star$, is the closed centered Euclidean ball with same volume as $E$.

Let $f:\R^n\to\R$ be a non-negative measurable function with super-level sets $\{f \geq t\}$ of finite measure.
The layer cake formula states that
\begin{equation}\label{eq_layer_cake}
f(x) = \int_0^\infty \chi_{\{f\geq t\}}(x) \d t
\end{equation}
for almost every $x \in \R^n$ and allows us to recover the function from its super-level sets.
The Schwarz symmetral of $f$, denoted by $f^\star$, is defined by
\[f^\star(x) = \int_0^\infty \chi_{\{f\geq t\}^\star}(x) \d t\]
for $x\in\R^n$.
Hence, $f^\star$ is determined by the properties of being radially symmetric, decreasing and having super-level sets of the same measure as those of $f$. Note that $f^\star$ is also called the symmetric decreasing rearrangement of $f$. 

\goodbreak
The proofs of our results make use of the Riesz rearrangement inequality, which is stated in full generality, for example, in \cite{BLL}. 

\begin{theorem}[Riesz's rearrangement inequality]\label{thm_BLL}
For $f,g,k:\R^n \to \R$ non-negative, measurable functions with super-level sets of finite measure,
    \[ \int_{\R^n}\int_{\R^n}  f(x) k(x-y) g(y) \d x \d y \leq \int_{\R^n}\int_{\R^n}  f^\star(x) k^\star(x-y) g^\star(y) \d x \d y. \]
\end{theorem}

We will  use the characterization of equality cases of the Riesz rearrangement inequality due to Burchard  \cite{Burchard96}.

\begin{theorem}[Burchard]\label{thm_burchard}
    Let $A,B$ and $C$ be sets of finite positive measure in $\,\R^n$ and denote by $\alpha, \beta$ and $\gamma$ the radii of their Schwarz symmetrals $A^\star, B^\star$ and $C^\star$.
    For $\,\vert\alpha - \beta\vert < \gamma < \alpha+\beta$, there is equality in
    \[ \int_{\R^n}\int_{\R^n} \chi_A(y) \chi_B(x-y) \chi_C(x) \d x \d y \leq  \int_{\R^n}\int_{\R^n} \chi_{A^\star}(y) \chi_{B^\star}(x-y) \chi_{C^\star}(x) \d x \d y \]
    if and only if, up to sets of measure zero,
    \[A=a+\alpha D,\, B = b+\beta D,\, C = c+\gamma D,\]
    where $D$ is a centered ellipsoid, and $a,b$ and $c=a+b$ are vectors in $\R^n$.
\end{theorem}

\subsection{Star-shaped sets and star bodies}
A set $K \subseteq \R^n$ is star-shaped (with respect to the origin), if the interval $[0,x]\subset K$ for every $x\in K$. 
The gauge function $\|\cdot\|_K : \R^n \to [0,\infty]$ of a star-shaped set is defined as
    \[\|x\|_K = \inf\{ \lambda > 0 : x \in \lambda K\},\]
and the radial function $\rho_K:\R^n \setminus\{0\} \to [0,\infty]$ as
    \[\rho_K(x) = \|x\|_K^{-1} = \sup\{\lambda \geq 0:\lambda x \in K\}.\]
The $n$-dimensional Lebesgue measure or volume of a star-shaped set $K$ in $\R^n$ with measurable radial function is given by
    \[\vol{K}= \frac1 n \int_{\sn} \rho_K(\xi)^n\d \xi.\]
We call a star-shaped set $K\subset\R^n$ a star body if its radial function is strictly positive and continuous in $\R^n \setminus \{0\}$. On the set of star bodies, the $q$-radial sum for $q\ne0$ of $K,L\subset\R^n$ is defined by
\begin{equation}
   \rho^q(K\ra{q} L,\xi)= \rho^q(K,\xi)+\rho^q(L,\xi)
\end{equation}
for $\xi\in\sn$ (cf.~\cite[Section 9.3]{Schneider:CB2}).  The dual Brunn--Minkowski inequality  (cf.~\cite[(9.41)]{Schneider:CB2}) states that for star bodies $K,L\subset\R^n$ and $q>0$,
\begin{align}\label{eq_dualBM}
    \vol{K\ra{-q} L}^{-q/n} 
    \ge  \vol{K}^{-q/n} +\vol{L}^{-q/n},
\end{align}
with equality precisely if $K$ and $L$ are dilates, that is, there is $\lambda>0$ such that $K=\lambda L$.

\goodbreak
Let $\alpha\in\R\backslash\{0,n\}$. For star-shaped sets $K,L \subseteq \R^n$  with measurable radial functions, the dual mixed volume is defined as
    \[\tilde V_\alpha(K,L) = \frac 1n \int_{\sn} \rho_K(\xi)^{n-\alpha} \rho_L(\xi)^{\alpha} \d \xi.\]
Notice that 
    \[\tilde V_\alpha(K,K) = \vol{K}\]
and that
\[\tilde V_\alpha(K,L_1\ra{\alpha} L_2)=\tilde V_\alpha(K, L_1)+\tilde V_\alpha(K,L_2)\]
for star-shaped sets $K, L_1, L_2\subseteq \R^n$ with measurable radial functions.

For star-shaped sets $K,L \subseteq \R^n$ of finite volume and $0<\alpha<n$, the dual mixed volume inequality states that 
\begin{equation}\label{eq_mixedvolume}
    \tilde V_{\alpha}(K,L) \leq \vol{K}^{({n-\alpha})/n} \vol{L}^{ \alpha/n}.
\end{equation}
Equality holds if and only if $K$ and $L$ are dilates, where we say that star-shaped sets $K$ and $L$ are dilates if $\rho_K=\lambda\,\rho_L$ almost everywhere on $\sn$ for some $\lambda> 0$. The definition of dual mixed volume for star bodies is due to Lutwak \cite{Lutwak75}, where also the dual mixed volume inequality is derived from H\"older's inequality (also see \cite[Section~9.3]{Schneider:CB2} or \cite[B.29]{Gardner}). 

\subsection{Anisotropic fractional Sobolev norms}

Let $0<s<1$ and $p\ge 1$. For $K\subset \R^n$ a star body and $f \in W^{s,p}(\R^n)$, the anisotropic fractional $L^p$ Sobolev norm of $f$ with respect to $K$ is 
\begin{equation}\label{eq_aniso_norm}
    \int_{\R^n}\int_{\R^n} \frac{\vert f(x)-f(y)\vert^p}{\Vert x-y\Vert_K^{n+p s}} \d x\d y.
\end{equation}
It was introduced in \cite{Ludwig:fracnorm} for $K$ a convex body (also, see \cite{Ludwig:fracperi}). For $K=\B$, the Euclidean unit ball, we obtain the classical $s$-fractional $L^p$ Sobolev norm of $f$. The limit as $s\to1^-$ was determined in \cite{BourgainBrezisMironescu} in the Euclidean case and in \cite{Ludwig:fracnorm} in the anisotropic case. We will also consider the following asymmetric versions of \eqref{eq_aniso_norm},
    \[\int_{\R^n}\int_{\R^n} \frac{\plus{ f(x)-f(y)}p}{\Vert x-y\Vert_K^{n+p s}} \d x\d y, \quad
    \int_{\R^n}\int_{\R^n} \frac{\minus{ f(x)-f(y)}p}{\Vert x-y\Vert_K^{n+p s}} \d x\d y,\]
where $\pl a = \max\{a,0\}$  and $\mn a = \max\{-a, 0\}$ for $a\in\R$. The limits as $s\to1^-$ were determined in \cite{Ma2014}.

\subsection{$L^p$ polar projection bodies}\label{sec_lp_polar}

For $p\ge 1$ and $f \in W^{1,p}(\R^n)$, the $L^p$ polar projection body is defined  as the star body with gauge function given by
\begin{equation}
    \Vert \xi\Vert_{\opp p f}^p  = \int_{\R^n}  \vert\langle \nabla f(x),\xi \rangle\vert^p \d x
\end{equation}
for $\xi\in\sn$, were $\langle \cdot ,\cdot\rangle$ denotes the inner product. It is the polar body of a convex body. The definition is due to Lutwak, Yang, and Zhang \cite{LYZ2002b}. For a convex body $K\subset \R^n$, they defined the $L^p$ polar projection body (with a different normalization) in \cite{LYZ2000} by
\begin{equation}\label{eq_defop}
    \Vert \xi\Vert_{\opp p K}^p  = \int_{\sn} \vert \langle \xi, \eta\rangle\vert^p \d S_p(K,\eta),
\end{equation}
where $S_p(K,\cdot)$ is the $L^p$ surface area measure of $K$ (for the definition of $L^p$ surface area measures, see, for example, \cite[Section~9.1]{Schneider:CB2}).

Asymmetric $L^p$ polar projection bodies of convex bodies were introduced in \cite{Ludwig:Minkowski}. For $f \in W^{1,p}(\R^n)$, the asymmetric $L^p$ polar projection bodies of $f$ are defined  as the star bodies with gauge function given by 
\begin{equation}
    \Vert \xi\Vert_{\appm p f}^p  = \int_{\R^n}  \plmn{\langle \nabla f(x),\xi \rangle}^p \d x
\end{equation}
for $\xi\in\sn$.

\section{Fractional $L^p$ Polar Projection Bodies}
\label{sec_fracprojectionbody}
Let $0<s<1$ and $1< p  <n/s$.
For $f \in W^{s,p}(\R^n)$, define the $s$-fractional $L^p$~polar projection body $\fpp s p f$ as the star-shaped set given by the gauge function
\begin{equation}\label{eq_fpp_def}
	\Vert \xi\Vert_{\fpp s p f}^{p s} = \int_0^\infty t^{-p s-1} \int_{\R^n} \left| f(x+t\xi)-f(x)\right|^p \d x \d t
\end{equation}
for $\xi\in\R^n$. Note that $\Vert \cdot\Vert_{\fpp s p f}$ is a one-homogeneous function on $\R^n$. 

Let $K\subset\R^n$ be a star body. The following simple calculation turns out to be useful. For $f\in W^{s,p}(\R^n)$,
\begin{align}
	\int_{\R^n} \int_{\R^n} &\frac{|f(x)-f(y)|^p}{\|x-y\|_K^{n+p s}} \d x \d y\\
	&= \int_{\R^n} \int_{\R^n} \frac{|f(y+z)-f(y)|^p}{\|z\|_K^{n+p s}} \d z \d y\\
	&= \int_{\sn} \int_0^\infty {\|t \xi\|_K^{-n-p s}} \int_{\R^n} {|f(y+t\xi)-f(y)|^p}\,t^{n-1}  \d y \d t \d\xi\\
	&= \int_{\sn} \int_0^\infty {\|\xi\|_K^{-n-p s}} t^{-p s-n} \left\|f(\cdot+t\xi) - f\right\|_{p}^p t^{n-1} \d t \d\xi\\
	&= \int_{\sn} \rho_K(\xi)^{n+p s} \int_0^\infty t^{-p s-1} \left\|f(\cdot+t\xi) - f\right\|_{p}^p \d t \d\xi\\
	&= \int_{\sn} \rho_K(\xi)^{n+p s} \rho_{\fpp s p f}(\xi)^{-p s} d\xi.
\end{align}
Hence,
\begin{equation}\label{eq_dualmixed_func}
	\int_{\R^n} \int_{\R^n} \frac{|f(x)-f(y)|^p}{\|x-y\|_K^{n+p s}} \d x \d y = n\, \tilde V_{-p s}(K, \fpp s p f)
\end{equation}
in this case.

Next, we establish basic properties of fractional $L^p$ polar projection bodies.

\begin{proposition}\label{prop_properties}
	For non-zero $f \in W^{s,p}(\R^n)$, the set $\,\fpp s p f$ is an origin-symmetric star body with the origin in its interior.
	Moreover, there is $c>0$ depending only on $f$ and $p$ such that $\fpp s p f \subseteq c\, \B$ for every $s \in (0,1)$.
\end{proposition}

\begin{proof}
First, note that since for $\xi\in\R^n$ and $t>0$,
	\[\int_{\R^n}\vert {f(x -t\xi)-f(x)}\vert^p \d x=\int_{\R^n}\vert f(x)-f(x+t\xi)\vert^p \d x,\]
the set $\fpp s p f$ is origin-symmetric.

Next, we show that $\fpp s p f$ is bounded.
We take $r>1$ large enough so that $\|f\|_{L^p(r\B)} \geq \frac 23 \|f\|_p$ and easily see that for $t>2r$,
\begin{align}
	\Vert f(\cdot+t\xi) -  f(\cdot)\Vert_p 
	&\geq  \|f(\cdot+t\xi) - f(\cdot)\|_{L^p(r\B-t\xi)}\\
	&=  \|f(\cdot) - f(\cdot-t\xi)\|_{L^p(r\B)}\\
	&\geq  \|f\|_{L^p(r\B)} - \|f(\cdot - t\xi)\|_{L^p(r\B)}\\
	&\geq  \frac 23\|f\|_p - \frac 13 \|f\|_p. 
\end{align}
Hence,
\begin{align}
	\int_0^\infty t^{-p s-1} \int_{\R^n} \big| {f(x+t\xi)-f(x)} \big|^p \d x \d t
	\geq \frac {\|f\|_p^p} {3^p}\, \, \int_r^\infty t^{-p s-1} \d t
	\geq \frac {\|f\|_p^p} {3^p} \, \frac {r^{-p s}}{p s} \geq c,
\end{align}
which implies that $\fpp s p f \subseteq c\, \B$ for $c>0$ independent of $s$.

\goodbreak
Now, we show that $\fpp s p f$ has the origin in its interior.
First observe that for $\xi, \eta \in \R^n$, by the triangle inequality and a change of variables,
\begin{align}
	\Vert \xi+&\eta\Vert _{\fpp s p f}^{p s}\\
	&= \int_0^\infty \!t^{-p s-1} \|f(\cdot+t\xi+t\eta)-f(\cdot)\|_p^p \d t \\
	&\leq \int_0^\infty \!t^{-p s-1} \left( \|f(\cdot+t\xi+t\eta)-f(\cdot+t\xi)\|_p + \|f(\cdot+t\xi)-f(\cdot)\|_p\right)^p \d t \label{eq_sublinearity}\\
	&\leq \int_0^\infty \!t^{- p s-1} 2^{p-1} (\|f(\cdot+t\eta)-f(\cdot)\|_p^p + \|f(\cdot+t\xi)-f(\cdot)\|_p^p) \d t\\
	&= 2^{p-1} \|\xi\|_{\fpp s p f}^{p s} + 2^{p-1} \|\eta\|_{\fpp s p f}^{p s}. 
\end{align}
Using the relation \eqref{eq_dualmixed_func} with $K=\B$, we get
\begin{equation}\label{eq_finite}
    \int_{\sn} \Vert\xi\Vert_{\fpp s p f}^{ p s} \d\xi
     = \frac 1n \int_{\R^n}\int_{\R^n} \frac{|f(x)-f(y)|^p}{|x-y|^{n+p s}} \d x \d y,
\end{equation}
which is finite since $f\in W^{s,p}(\R^n)$.
We choose $r>0$ large enough so that the set $A = \{\xi \in \sn: \Vert\xi\Vert_{\fpp s p f}^{s} < r\}$ has positive $(n-1)$-dimensional Hausdorff measure and contains a basis $\{\xi_1, \ldots, \xi_n\} \subseteq A$ of $\R^n$.
Applying (if necessary) a linear transformation to $\fpp s p f$ , we may assume without loss of generality that $\xi_i = e_i$ are the canonical basis vectors.
For every $x \in \R^n$, writing $x = \sum x_i e_i$ and using \eqref{eq_sublinearity}, we get
\begin{equation}\label{eq_euclideanbound}
	\|x\|_{\fpp s p f} \leq \Big( 2^{n(p-1)} \sum_{i=1}^n |x_i|^{p s} \|e_i\|_{\fpp s p f}^{p s} \Big)^{\frac 1{p s}} \leq d\, |x|,
\end{equation}
where $d>0$ is independent of $x$.
This shows that $\fpp s p f$ has the origin as interior point.

Finally, we show that $\| \cdot \|_{\fpp s p f}$ is continuous.
For $\xi, \eta \in \R^n$, by the triangle inequality and \eqref{eq_euclideanbound}, we have
\begin{align}
    \|\xi + &\eta\|_{\fpp s p f}^{p s}\\[2pt]
    &= \int_0^\infty t^{-1- p s} \|f(\cdot + t \xi + t \eta) - f(\cdot)\|_p^p\d t\\
    & \leq \int_0^\infty t^{-1- p s} \big( \|f(\cdot + t \eta) - f(\cdot)\|_p + \|f(\cdot + t \xi) - f(\cdot)\|_p \big)^p \d t \\
    & \leq \big(1+|\eta|^{\frac s2\frac{p}{p-1}}\big)^{p-1}\int_0^\infty {t^{-1- p s}}\left(\frac{\|f(\cdot + t \eta) - f(\cdot)\|_p^p}{|\eta|^{\frac{p s}2}} + \|f(\cdot + t \xi) - f(\cdot)\|_p^p \right)\d t \\[2pt]
    & = \big(1+|\eta|^{\frac s2\frac{p}{p-1}}\big)^{p-1} \big( |\eta|^{-\frac{p s}2}\|\eta\|_{\fpp s p f}^{p s} + \|\xi\|_{\fpp s p f}^{p s} \big)\\[4pt]
    & \leq \big(1+|\eta|^{\frac s2\frac{p}{p-1}}\big)^{p-1} \big( d\, |\eta|^{\frac{p s}2} + \|\xi\|_{\fpp s p f}^{p s} \big),
\end{align}
where we used the inequality $a+b \leq (1+r^{p/(p-1)})^{(p-1)/p}((r^{-1} a)^p + b^p)^{1/p}$ for $a,b,r > 0$, which is a consequence of H\"older's inequality.

\goodbreak
We obtain 
\begin{equation}\label{eq_quasiineq_sup}
	\|\xi + \eta\|_{\fpp s p f}^{p s} \leq \big(1+|\eta|^{\frac s2\frac{p}{p-1}}\big)^{p-1} \big( d\, |\eta|^{\frac{p s}2} + \|\xi\|_{\fpp s p f}^{p s} \big).
\end{equation}
Applying inequality \eqref{eq_quasiineq_sup} to the vectors $\xi + \eta$ and $-\eta$, we get
\begin{equation}
	\|\xi\|_{\fpp s p f}^{p s} = \|\xi + \eta - \eta\|_{\fpp s p f}^{p s} \leq \big(1+|-\eta|^{\frac s2\frac{p}{p-1}}\big)^{p-1} \big( d\, |-\eta|^{\frac{p s}2} + \|\xi + \eta\|_{\fpp s p f}^{p s} \big),
\end{equation}
which implies
\begin{equation}\label{eq_quasiineq_inf}
	\|\xi + \eta\|_{\fpp s p f}^{p s} \geq \big(1+|\eta|^{\frac s2\frac{p}{p-1}}\big)^{p-1} \|\xi\|_{\fpp s p f}^{p s} - d\, |\eta|^{\frac{p s}2}.
\end{equation}
The continuity of $\|\cdot\|_{\fpp s p f}$ now follows from \eqref{eq_quasiineq_sup} and \eqref{eq_quasiineq_inf}.
\end{proof}

\section{Fractional Asymmetric $L^p$ Polar Projection Bodies}\label{sec_fracaprojectionbody}
Let $0<s<1$ and $1< p  <n/s$. For $f \in W^{s,p}(\R^n)$, define the asymmetric $s$-fractional $L^p$ polar projection bodies $\ppp s p f$ and $\ppm s p f$ as the star-shaped sets given by the gauge functions
\begin{equation}	
	\Vert \xi \Vert_{\pppm s p f}^{p s} = \int_0^\infty t^{-p s-1} \int_{\R^n} \plusminus{f(x+t\xi)-f(x)}p \d x \d t
\end{equation}
for $\xi\in\R^n$. 
We have $\ppm s p f= \ppp s p {(-f)}=-\ppp s p f$ and state our results just for $\ppp s p f$. Note that, as in the symmetric case,  $\Vert \cdot \Vert_{\ppp s p f}^{p s}$ is a one-homogeneous function on $\R^n$. Also note that
\begin{equation}\label{eq_plusminus}   
    \Vert\xi\Vert_{\fpp s p f}^{p s} =\Vert \xi\Vert_{\ppp s p f}^{p s}+ \Vert\xi\Vert_{\ppm s p f}^{p s}
\end{equation} 
for $\xi\in\R^n$.

Let $K\subset \R^n$ be a star body and $f \in W^{s,p}(\R^n)$. 
As in \eqref{eq_dualmixed_func}, we obtain that
\begin{equation}\label{eq_dualmixed_func+}
	\int_{\R^n} \int_{\R^n} \frac{\plus{f(x)-f(y)}p}{\|x-y\|_K^{n+p s}} \d x \d y = n\, \tilde V_{-p s}(K, \ppp s p f). 
\end{equation}
In the following proposition, we derive the basic properties of fractional asymmetric $L^p$ polar projection bodies.

\begin{proposition}\label{prop_aproperties}
	For non-zero $f \in W^{s,p}(\R^n)$, the set $\,\ppp s p f$ is a star body with the origin in its interior.
	Moreover, there is $c>0$ depending only on $f$ and $p$ such that $\ppp s p f \subseteq c\, \B$ for every $s \in (0,1)$.
\end{proposition}
\begin{proof}
Since the functions $\plus{a}p$ and $\minus{a}p$ are convex, the inequalities $\plus{a+b}p\geq \plus{a}p + p \plus a{p-1} b$ and $\minus{a+b}p \geq \minus a p + p \minus a{p-1} b$ hold for $a,b\in\R$.

If $\int_{\R^n} \plus{f(x)}p \d x>0$, take $\varepsilon >0$ so small that $\varepsilon +p \varepsilon ^{1/p} \left\| f\right\| _p^{p-1}\leq \frac 12 \int_{\R^n} \plus{f(x)}p \d x$, and take $r>0$ so large that $\int_{\R^n \setminus r\B} |f(x)|^p \d x < \varepsilon$.
For $z \in \R^n \setminus 2r \B$, we obtain by H\"older's inequality that
\begin{align}
	\int _{r\B}\plus{f(x)&-f(x+z)}p \d x \\
	&\geq \int _{r\B} \plus{f(x)}p-p\, \plus{f(x)}{p-1} f(x+z) \d x \\
	&\geq \int _{r\B} \plus{f(x)}p \d x - p \Big(\int _{r\B} \plus{f(x)}p \d x \Big)^{\frac{p-1}{p}} \Big(\int _{r\B}\left| f(x+z)\right| ^p \d x\Big)^{\frac{1}{p}}\\
	&\geq \int _{r\B} \plus{f(x)}p \d x - p \Big(\int _{\R^n} |f(x)|^p \d x \Big)^{\frac{p-1}{p}} \Big(\int _{\R^n\setminus r\B}\left| f(x)\right| ^p \d x\Big)^{\frac{1}{p}} \\
	&\geq \int _{\R^n} \plus{f(x)}p \d x - \varepsilon - p \,\|f\|_p^{p-1} \varepsilon^{\frac1p}\\
	&\geq \frac 12 \int _{\R^n} \plus{f(x)}p \d x.
\end{align}
In case $\int_{\R^n} \plus{f(x)}p \d x=0$ the previous inequality holds trivially for any $r>0$.
	
By an analogous calculation, and eventually increasing the  value of $r$, we obtain that
\begin{align}
	\int _{r\B -z }\plus{f(x)-f(x+z)}p \d x
	&= \int _{r\B}\minus{f(x)-f(x-z)}p \d x\\
	&\geq \frac 12 \int _{\R^n} \minus{f(x)}p \d x.
\end{align}
It follows that $\int_{\R^n} \plus{f(x)-f(x+z)}p \d x \geq \frac 12 \|f\|_p^p$ for every $z \in \R^n \setminus 2r \B$ with $r>0$ depending only on $f$.
Finally,
\begin{align}
	\|\xi\|_{\ppp s p f} ^{p s}
	&\geq \int_{2r}^\infty t^{-1-p s} \int_{\R^n} \plus{f(x)-f(x+z)}p \d x \d t\\
	&\geq \int_{2r}^\infty t^{-1-p s} \d t\, \frac 12 \int _{\R^n} |f(x)|^p \d x \\
	&\geq \frac{(2r)^{-p s}}{p s} \frac 12 \int _{\R^n} |f(x)|^p \d x \\
	&\geq \frac{(2r)^{-p}}{2 p} \|f\|_p^p.
\end{align}

Note that $\fpp s p f\subset \ppp s p f$. Hence, it follows from Proposition \ref{prop_properties} that $\ppp s p f$ contains the origin in its interior, that is, there is $d>0$ such that
\begin{equation}\label{eq_euclideanbounda}
	\|x\|_{\ppp s p f} \leq d\, |x|
\end{equation}
for every $x\in\R^n$.

Finally, we show that $\| \cdot \|_{\ppp s p f}$ is continuous. Observe that the inequality $(a+b)_+^p \leq (a_+ + b_+)^p$ holds for any $a,b \in \R$.
Hence, for $\xi, \eta \in \R^n$, we obtain that
\begin{align}
    &\int_{\R^n} \plus{f(x + t \xi + t \eta) - f(x) }p \d x  \\
    &=\int_{\R^n} \plus{f(x + t \xi + t \eta) - f(x + t \xi) + f(x + t \xi) - f(x) }p \d x  \\
    &\leq  \int_{\R^n} \big( \pl{(f(x + t \xi + t \eta) - f(x + t \xi))} + \pl{(f(x + t \xi) - f(x))} \big)^p \d x  \\
    &\leq  \int_{\R^n} (1+|\eta|^{\frac s 2\frac {p }{(p-1)} })^{p-1} \Big( \frac{ \plus{f(x + t \xi + t \eta) - f(x + t \xi)}p}{|\eta|^{\frac {p s}2}} + \plus{f(x + t \xi) - f(x)}p\Big) \d x \\
    & \leq (1+|\eta|^{\frac s2\frac {p }{(p-1)} })^{p-1} \Big( \frac{\|\pl{(f(\cdot + t \eta) - f(\cdot))}\|_p^p}{|\eta|^{\frac {p s}2}} +  \|\pl{(f(\cdot + t \xi) - f(\cdot))}\|_p^p \Big),
\end{align}
where we used the inequality $a+b \leq (1+r^{ p/(p-1)})^{(p-1)/p}((r^{-1} a)^p + b^p)^{1/p}$ for $a,b,r > 0$, which is a consequence of H\"older's inequality. Thus, integrating and using \eqref{eq_euclideanbounda}, we obtain
\begin{equation}\label{eq_quasiineq_sup+}
	\|\xi + \eta\|_{\ppp s p f}^{p s} \leq (1+|\eta|^{\frac s 2\frac p{p-1} })^{p-1} ( d\, |\eta|^{\frac{p s}2} + \|\xi\|_{\ppp s p f}^{p s} ).
\end{equation}
Applying inequality \eqref{eq_quasiineq_sup+} to the vectors $\xi + \eta$ and $-\eta$, we get
\begin{equation}
	\|\xi\|_{\ppp s p f}^{p s} = \|\xi + \eta - \eta\|_{\ppp s p f}^{ p s} \leq (1+|-\eta|^{\frac s 2\frac p{p-1}})^{p-1} ( d\, |-\eta|^{\frac{p s}2} + \|\xi + \eta\|_{\ppp s p f}^{p s} ),
\end{equation}
which implies
\begin{equation}\label{eq_quasiineq_inf+}
	\|\xi + \eta\|_{\ppp s p f}^{p s} \geq (1+|\eta|^{\frac s 2\frac p{p-1}})^{-(p-1)} \|\xi\|_{\ppp s p f}^{ p s} - d\, |\eta|^{\frac{p s}2}.
\end{equation}
The continuity of $\|\cdot\|_{\ppp s p f}$ now follows from \eqref{eq_quasiineq_sup+} and \eqref{eq_quasiineq_inf+}.
\end{proof}

\section{The Limit of Fractional $L^p$ Polar Projection Bodies}\label{sec_limit2}
We establish the limiting behavior of $s$-fractional $L^p$ polar projection bodies for $1< p  <n/s$ as $s\to1^-$ in the symmetric and asymmetric case. 
For $p=1$, a corresponding result was proved in \cite{HaddadLudwig_fracsob}.

Let $0<s<1$ and $1< p  <n/s$. Set $p'=p/(p-1)$. We say that $f_k \to f$ weakly in $L^p(\R^n)$ if
	\[\int_{\R^n} f_k(x) g(x) \d x \to \int_{\R^n} f(x) g(x) \d x\]
for every $g \in L^{p'}(\R^n)$ as $k\to\infty$. Set $\Bp = \{g \in L^{p'}(\R^n) : g \geq 0, \|g\|_{p'} \leq 1 \}$.

\goodbreak
We require the following lemmas.

\begin{lemma}\label{lem_positivepart}
    The following statements hold.
	\begin{enumerate}[label={(\arabic*)}, ref = \ref{lem_positivepart}\,(\arabic*)]
	\item\label{lem_dualnorm}
	For $f \in L^p(\R^n)$,  
	\[\|\pl{f}\|_p = \sup_{g \in \Bp} \int_{\R^n} f(x) g(x) \d x.\]
	  \item\label{lem_weaklimit}
	Let $f_k, f\in L^p(\R^n)$. If $f_k \to f$ weakly in $L^p(\R^n)$ as $k\to\infty$, then
	\[\liminf_{k\to\infty} \|\pl{(f_k)}\|_p \geq \|\pl{f}\|_p.\]
    \item\label{lem_density}
	Assume $f_k$ is a bounded sequence in $L^p(\R^n)$. If
	\[\lim_{k\to\infty}\int_{\R^n} f_k(x) g(x)\d x = \int_{\R^n} f(x) g(x)\d x\]
	for every $g$ in a dense subset $D \subseteq L^{p'}(\R^n)$, then $f_k \to f$ weakly in $L^p(\R^n)$ as $k\to\infty$.
	\end{enumerate}
\end{lemma}

\begin{proof}
First we prove (1). Let $g \in \Bp$ and write $f = \pl{f} - \mn{f}$. Since $\mn{f}$ and $g$ are non-negative, it follows from H\"older's inequality that
	\[ \int_{\R^n} f(x) g(x)\d x 
	\leq \int_{\R^n} \pl{f}(x) g(x) \d x\leq \|\pl{f}\|_p. \]
For the opposite inequality, take $g = \|\pl{f}\|_p^{-{p}/{p'}} \pl{f}^{{p}/{p'}}$ and notice that $g \in \Bp$ and
	\[ \int_{\R^n} f(x) g(x)\d x = \|\pl{f}\|_p^{-\frac p{p'}} \int_{\R^n} f(x) \pl{f}(x)^{\frac p{p'}} \d x \le \|\pl{f}\|_p^{-\frac p{p'}} \int_{\R^n} \pl{f}(x)^p\d x = \|\pl{f}\|_p.\]

Next we prove (2). Fix $k_0$ and $g_0 \in \Bp$. By (1), we  have
	\[\int_{\R^n}  f_{k_0}(x) g_0(x) \d x \leq \sup_{g \in \Bp} \int_{\R^n} f_{k_0}(x) g(x)\d x = \|\pl{(f_{{k_0}})}\|_p.\]
Since this inequality holds for every $k_0$,
	\[\int_{\R^n} f(x) g_0(x) \d x = \lim_{k\to\infty} \int_{\R^n}  f_k(x) g_0(x)\d x \leq \liminf_{k\to\infty} \|\pl{(f_{k})}\|_p.\]
Thus, by (1),
	\[ \|\pl{f}\|_p = \sup_{g \in B_{p',+}} \int_{\R^n} f(x)g(x)\d x \leq \liminf_{k\to\infty} \|\pl{(f_{k})}\|_p.\]

Finally, we prove (3).	Take $c \ge \max\{\|f_k\|_p, \|f\|_p\}$.
Let $\varepsilon > 0$ and $g \in L^{p'}(\R^n)$. Take $h \in D$ such that $\|g-h\|_{p'} < \varepsilon/(2c)$. Then
\begin{align}
	&\Big| \int_{\R^n} f_k(x) g(x)\d x - \int_{\R^n} f(x) g(x)\d x \Big|\\
	&\leq \Big| \int_{\R^n} f_k(x) (g(x) - h(x))\d x \Big| + \Big| \int_{\R^n} f_k(x) h(x) \d x - \int_{\R^n} f(x) h(x)\d x \Big|\\
	& \phantom{\leq} + \Big| \int_{\R^n} f(x) (g(x) - h(x)) \d x \Big|\\
	&\leq c \varepsilon/(2c) + \Big| \int_{\R^n} f_k(x) h(x)\d x - \int_{\R^n} f(x) h(x) \d x \Big| + c \varepsilon/(2c)
\end{align}
and the statement follows.
\end{proof}

\goodbreak
\begin{lemma}\label{lem_limitaL1}
	For $f \in W^{1,p}(\R^n)$ and fixed $\xi \in \sn$,
	   \[\lim_{t\to 0} \Big\| \pl{\Big( \frac{f(\cdot+t\xi)-f(\cdot)}t \Big)} \Big\|_{p}^p = \int_{\R^n} \pl{\langle \nabla f(x), \xi \rangle}^p \d x.\]
\end{lemma}

\begin{proof}
Let $g:\R^n\to \R$ be a smooth function with compact support. Write $\div_x$ for the divergence taken with respect to the variable $x$. Using integration by parts, we obtain for $\xi\in\sn$ and $t>0$,
\begin{align*}
	\int_{\R^n} g(x) \frac{f(x+t\xi)-f(x)}t \d x
	&= \int_{\R^n} f(x) \frac{g(x-t\xi)-g(x)}t \d x\\
	&= - \int_{\R^n} f(x) \int_0^1 \langle \nabla g(x-r t\xi), \xi \rangle \d r \d x\\
	&= - \int_{\R^n} f(x) \div_x\Big(\int_0^1 g(x-r t\xi) \d r\, \xi \Big)\d x \\
	&= \int_{\R^n} \Big(\int_0^1 g(x-r t\xi) d r \Big) \langle \nabla f(x),\xi \rangle \d x.
\end{align*}
By Minkowski's integral inequality $\| \int_0^1 g(\cdot-r t\xi) \d r \|_{p'} \leq \|g\|_{p'}$, and we deduce 
	\[ \Big\|\frac{f(\cdot+t\xi)-f(\cdot)}t \Big\|_p \leq \| \langle \nabla f(\cdot),\xi \rangle \|_p < \infty.\]
Hence, $\frac{f(\cdot+t\xi)-f(\cdot)}t$ is uniformly bounded in $L^p(\R^n)$ on $(0,\infty)$.

By Lemma \ref{lem_density},
\begin{align}
	\label{eq_weaklimit}
	\lim_{t\to 0}\int_{\R^n} g(x) \frac{f(x+t\xi)-f(x)}t \d x
	= \int_{\R^n} g(x) \langle \nabla f(x), \xi \rangle \d x
\end{align}
for every $g \in L^{p'}(\R^n)$. 
Hence, $\frac{f(\cdot+t\xi)-f(\cdot)}t$ converges weakly to $\langle \nabla f(\cdot), \xi \rangle$ as $t\to 0$.

By Lemma \ref{lem_weaklimit},
	\[\liminf_{t\to 0} \Big\|\pl{\Big( \frac{f(\cdot+t\xi)-f(\cdot)}t \Big)} \Big\|_{p}	\geq \|\pl{\langle \nabla f(\cdot), \xi \rangle} \|_p.\]
For the opposite inequality we recall that for any $g \in \Bp$, the function $x\mapsto \int_0^1 g(x-r t\xi) \d r$ is in $\Bp$ as well.
Hence, 
\begin{align}
    \int_{\R^n} g(x) \frac{f(x+t\xi)-f(x)}t \d x
	&= \int_{\R^n} \Big(\int_0^1 g(x-r t\xi) \d r \Big) \langle \nabla f(x),\xi \rangle \d x \\
   &\leq \|\pl{\langle \nabla f(x), \xi \rangle} \|_p.
\end{align}
Again by Lemma \ref{lem_dualnorm},
	\[\Big\|\pl{\Big( \frac{f(\cdot+t\xi)-f(\cdot)}t \Big)} \Big\|_{p} \leq \|\pl{\langle \nabla f(\cdot), \xi \rangle} \|_p\]
for each $t>0$.
\end{proof}

The following result is Lemma 4 in \cite{HaddadLudwig_fracsob}.

\begin{lemma}\label{lem_limit1D}
	If $\varphi:[0,\infty) \to [0,\infty)$ be a measurable function with $\lim_{t\to 0^+} \varphi(t) = \varphi(0)$ and such that $\int_0^\infty t^{-s_0} \varphi(t) \d t < \infty$ for some $s_0 \in (0,1)$, then
	\[ \lim_{s\to 1^-} (1-s) \int_0^\infty t^{-s} \varphi(t) \d t = \varphi(0).\]
\end{lemma}

\goodbreak
We are now able to prove the main result of this section.

\begin{theorem}\label{thm_limitaM}
	Let $f \in W^{1,p} (\R^n)$. For $\xi\in\sn$, 
	\[\lim_{s\to 1^-} (p(1-s))^{\frac 1p}\Vert\xi \Vert_{\ppp s p f} = \Vert\xi\Vert_{\app pf}.\]
	Moreover,
	\[ 
		\lim_{s\to 1^-} p(1-s) \vol{\ppp s p f}^{-\frac{p s}n} = \vol{\app pf}^{-\frac p n},
	\]
    and 
    \begin{equation}
	\lim_{s\to 1^-} p(1-s) \tilde V_{-p s}(K, \ppp{s}p f)=  
	\tilde V_{-p}(K,  \app pf)
    \end{equation}
    for every star body $K\subset \R^n$.
\end{theorem}

\begin{proof}
Define $\varphi: [0,\infty)\to [0,\infty)$ by
	\[\varphi(t) = \Big\| \pl{\Big( \frac{f(\cdot+t\xi)-f(\cdot)}t \Big)} \Big\|_p^p,\]
and note that $\varphi(t) \leq \left( \frac{2\|f\|_p}t \right)^p$ for $t> 0$.
By Lemma \ref{lem_limit1D} and Lemma \ref{lem_limitaL1},
	\[\lim_{s\to 1^-} p(1-s) \int_0^\infty t^{p(1-s)-1 } \Big\| \pl{\Big( \frac{f(\cdot+t\xi)-f(\cdot)}t \Big)} \Big\|_p^p \d t 	
	= \int_{\R^n} \pl{\langle \nabla f(x), \xi \rangle}^p\d x.\]
By Proposition \ref{prop_properties} we can use the dominated convergence theorem to obtain
\begin{align}
	\lim_{s\to 1^-}n\, &\vol{(p(1-s))^{-\frac 1{p s}} \ppp s p f}\\
	&= \lim_{s\to 1^-}\int_{\sn} \Big( p(1-s) \int_0^\infty t^{p(1-s)-1} \Big\|\pl{ \Big( \frac{f(\cdot+t\xi)-f(\cdot)}t \Big)} \Big\|_p^p \d t \Big)^{-\frac n{p s}} \d\xi \\
	&=  \int_{\sn} \Big( \int_{\R^n} \pl{\langle \nabla f(x),\xi \rangle} ^p \d x \Big)^{-\frac n p} \d\xi\\[4pt]
	&= n\, \vol{ \app pf},
\end{align}
and
\begin{align}
	\lim_{s\to 1^-} n p (1-s)\tilde V_{-p s}(K, \ppp{s} p f)
	&= \lim_{s\to 1^-}p(1-s)  \int_{\sn} \Vert \xi\Vert_K^{n+p s} \Vert\xi\Vert_{\ppp{s}p f}^{p s} \d\xi\\
    &=   \int_{\sn} \Vert \xi\Vert_K^{n} \,\Vert\xi\Vert_{\app p f}^p \d\xi\\
    &=  n\,\tilde V_{-p}(K, \app p f),
\end{align}
which completes the proof of the theorem.
\end{proof}

\goodbreak
The following result is an immediate consequence of Theorem \ref{thm_limitaM} and \eqref{eq_plusminus}.

\begin{theorem}\label{thm_limitM}
	Let $f \in W^{1,p} (\R^n)$. For $\xi\in\sn$, 
    \[\lim_{s\to 1^-} (p(1-s))^{\frac1p}\Vert\xi \Vert_{\fpp s p f} =  \Vert\xi\Vert_{\opp p f}.\]
	Moreover,
	\[ 
		\lim_{s\to 1^-} p(1-s) \vol{\fpp s p f}^{-\frac{p s}n} =  \vol{\opp pf}^{-\frac p n},
	\]
    and 
    \begin{equation}\label{eq_limitK}
	\lim_{s\to 1^-} p(1-s) \tilde V_{-p s}(K, \fpp{s}p f)=  
	\tilde V_{-p}(K,  \opp p f)
    \end{equation}
    for every star body $K\subset \R^n$.
\end{theorem}

\section{Anisotropic Fractional P\'olya--Szeg\H o Inequalities}

We will establish anisotropic P\'olya--Szeg\H o inequalities for fractional $L^p$ Sobolev norms and their asymmetric counterparts.

\begin{theorem}\label{thm_fburchard+}
 	If  $f\in L^p(\R^n)$ is non-negative and $K\subset \R^n$ a star body, then
	\begin{equation}\label{eq_fburchard+}
	\int_{\R^n}\int_{\R^n} \frac{\plus{f(x) - f(y)}p}{\Vert x-y\Vert_K^{n+p s}} \d x \d y \geq \int_{\R^n}\int_{\R^n} \frac{\plus{f^\star(x) - f^\star(y)}p}{\Vert x-y\Vert_{K^\star}^{n+p s}} \d x \d y.
	\end{equation}
Equality holds for non-zero $f\in W^{s, p}(\R^n)$ if and only if $K$ is a centered ellipsoid and $f$ is a translate of  $f^\star\circ \phi$ for some $\phi \in\sln$.
\end{theorem}

\begin{proof}
Writing 
    \[
    \|z\|_K^{-n-p s} = \int_0^\infty k_t(z) \d t
    \] 
where $k_t(z) = \chi_{t^{- 1/({n+p s})}K}(z)$,
we obtain
\begin{align}
    \int_{\R^n}\int_{\R^n} \frac{\plus{f(x) - f(y)}p}{\Vert x-y\Vert_K^{n+p s}} \d x \d y 
    = \int_0^\infty \int_{\R^n}\int_{\R^n} \plus{f(x) - f(y)}p k_t(x-y) \d x \d y \d t.
\end{align}
Note that
    \[
    \plus{f(x) - f(y)}{p} = p \int_{0}^{\infty} \plus{f(x)-r}{p-1} \chi_{\{f < r\}}(y) \d r.
    \]
Hence, for $t>0$, it follows from Fubini's theorem that
\begin{align}
    \int_{\R^n}\int_{\R^n} &\plus{f(x) - f(y)}{p}\, k_t(x-y) \d x \d y\\
    &=p \int_{0}^{\infty} \int_{\R^n} \int_{\R^n} \plus{f(x)-r}{p-1} k_t(x-y) \chi_{\{f < r\}}(y) \d x \d y \d r\\
    &=  p \int_{0}^{\infty} \int_{\R^n} \int_{\R^n} \plus{f(x)-r}{p-1} k_t(x-y) (1-\chi_{\{f \ge r\}}(y)) \d x \d y \d r.
\end{align}
Let $r,t>0$. Note that $ \int_{\R^n} \plus{f(x)-r}{p-1} \d x<\infty$ and that
\begin{align}
    &\int_{\R^n} \int_{\R^n} \plus{f(x)-r}{p-1} k_t(x-y) (1-\chi_{\{f \ge r\}}(y)) \d x \d y\\
    &=  p \,\Vert k_t\Vert_1  \int_{\R^n} \plus{f(x)-r}{p-1} \d x 
     - p \int_{\R^n} \int_{\R^n} \plus{f(x)-r}{p-1} k_t(x-y) \chi_{\{f \ge r\}}(y) \d x \d y.
\end{align}
The first term is finite since $\{f > r\}$ has finite measure, $f \in L^{\frac{n p}{n - p s}}(\R^n)$ and $\frac{n p}{n - p s} > p-1$.
Clearly the first term  is invariant under Schwarz symmetrization. For the second term, by the Riesz rearrangement inequality, Theorem \ref{thm_BLL}, we have
\begin{multline}\label{eq_BLLp}
    \int_{\R^n} \int_{\R^n} \plus{f(x)-r}{p-1} k_t(x-y) \chi_{\{f \ge r\}}(y) \d x \d y  \\
    \le
    \int_{\R^n} \int_{\R^n} \plus{f^\star(x)-r}{p-1} k_t^\star(x-y) \chi_{\{f^\star \ge r\}}(y) \d x \d y
\end{multline}
for $r, t>0$. Note that
    \[
    \plus{f(x) - r}{p-1} = (p-1) \int_{0}^{\infty} \plus{\tilde r-r}{p-2} \chi_{\{f \geq \tilde r\}}(x) \d \tilde r
    \]
and that the corresponding equation  holds for $f^\star$. Hence, if there is equality in \eqref{eq_fburchard+}, then, for  $(\tilde r,r, t)\in (0,\infty)^3\backslash M$ with $\vol{M}=0$, we have
\begin{equation}
\begin{aligned}
    \int_{\R^n} \int_{\R^n}   &\chi_{\{f \ge \tilde r\}}(x)  \chi_{t^{- 1/({n+p s})}K}(x-y) \chi_{\{f \ge r\}}(y) \d x \d y  \\
    &=
    \int_{\R^n} \int_{\R^n}  \chi_{\{f^\star \ge \tilde r\}}(x) \chi_{t^{- 1/({n+p s})}K^\star}(x-y) \chi_{\{f^\star \ge r\}}(y) \d x \d y.
\end{aligned}
\end{equation}
    
For almost every $(\tilde r, r)\in(0,\infty)^2$, we have $(\tilde r,r,t)\in (0,\infty)^3\backslash M$ for almost every $t>0$.
For  such $(\tilde r, r)$ with $\tilde r\le r$ and $t>0$ sufficiently large, the assumptions of Theorem \ref{thm_burchard} are fulfilled and therefore there are a centered ellipsoid $D$ and $a, b\in\R^n$ (depending on $(\tilde r, r, t)$)  such that
\begin{equation}
    \{f \ge \tilde r\}= a + \alpha D,\quad t^{-1/({n+p s})} K=  b + \beta D,\quad \{f \ge r\}= c + \gamma D
\end{equation}
where $c= a+ b$. 
Since $ K=  t^{1/({n+ps})}b + (\vol{K}/\vol{D})^{1/n} D$, the centered ellipsoid $D$ does not depend on $(\tilde r,r,t)$ and also $a,c$ do not depend on $t$. It follows that $b=0$ and  that $K$ is a multiple of $D$. Hence, $a=c$ is a constant vector which concludes the proof.
\end{proof}
 
 The following result is a variation of \cite[Theorem 3.1]{andreas}.

\begin{theorem}\label{thm_fburchard}
 	If  $f\in L^p(\R^n)$ is non-negative and $K\subset \R^n$ a star body, then
	\begin{equation}\label{eq_fburchard}
	\int_{\R^n}\int_{\R^n} \frac{|f(x) - f(y)|^p}{\Vert x-y\Vert_K^{n+p s}} \d x \d y \geq \int_{\R^n}\int_{\R^n} \frac{|f^\star(x) - f^\star(y)|^p}{\Vert x-y\Vert_{K^\star}^{n+p s}} \d x \d y.
	\end{equation}
    Equality holds for non-zero $f\in W^{s, p}(\R^n)$ if and only if $K$ is a centered ellipsoid and $f$ is a translate of  $f^\star\circ \phi$ for some $\phi \in\sln$.
\end{theorem}

\begin{proof}
Since 
    \[ \int_{\R^n}\int_{\R^n} \frac{\minus{f(x) - f(y)}p}{\Vert x-y\Vert_K^{n+p s}} \d x \d y=
    \int_{\R^n}\int_{\R^n} \frac{\plus{f(x) - f(y)}p}{\Vert x-y\Vert_{-K}^{n+p s}} \d x \d y,\]
the result follows from Theorem \ref{thm_fburchard+} for $K$ and $-K$.
\end{proof}

\section{Affine Fractional P\'olya--Szeg\H o Inequalities} 
 
We establish affine P\'olya--Szeg\H o inequalities for fractional asymmetric and symmetric $L^p$ polar projection bodies.

\begin{theorem}\label{thm_aPS+}
	If  $f\in W^{s,p}(\R^n)$ is non-negative, then
	\begin{equation}
	\vol{\ppp s p  f}^{-p s/n} \geq \vol{\ppp s p f^\star}^{- p s/n}.
	\end{equation}
    Equality holds if and only if $f$ is a translate of  $f^\star\circ \phi$ for some $\phi \in\sln$.
\end{theorem}

\begin{proof}         
By Theorem \ref{thm_fburchard+}, \eqref{eq_dualmixed_func+} and the dual mixed volume inequality, we obtain for $K\subset \R^n$ a star body that        
\begin{align}                
    \tilde V_{-p s} (K, \ppp s p f)                
    &\geq \tilde V_{-p s} (K^\star, \ppp s p  f^\star)\\\label{eq_aps+}                
    &\geq \vol{K^\star}^{({n+p s})/{n}} \vol{\ppp s p  f^\star}^{- p s/n}\\                
    &= \vol{K}^{({n+p s})/{n}} \vol{\ppp s p  f^\star}^{-  p s/n}.        
\end{align}        
Setting $K = \ppp s p f$, we  see that        
\begin{equation}                
    \vol{\ppp s p f}                
    = \tilde V_{-p s} (\ppp s p  f, \ppp s p f)
    \geq \vol{\ppp s p  f}^{({n+p s})/{n}} \vol{\ppp s p f^\star}^{-p s/n},        
\end{equation}        
which completes the proof of the inequality. By Theorem~\ref{thm_fburchard+}, there is equality in \eqref{eq_aps+} if and only if $f$ is a translate of  $f^\star\circ \phi$ for some $\phi \in\sln$.
\end{proof}

\goodbreak
The following result is obtained in the same way as Theorem \ref{thm_aPS+} by replacing Theorem~\ref{thm_fburchard+} with Theorem \ref{thm_fburchard}.

\begin{theorem}\label{thm_aPS}
	If  $f\in L^p(\R^n)$ is non-negative, then
	\begin{equation}
	\vol{\fpp s p  f}^{-p s/n} \geq \vol{\fpp s p f^\star}^{- p s/n}.
	\end{equation}
    Equality holds for $f\in W^{s, p}(\R^n)$ if and only if $f$ is a translate of  $f^\star\circ \phi$ for some $\phi \in\sln$.
\end{theorem}

We remark that by Theorem \ref{thm_limitM} we obtain from Theorem \ref{thm_aPS} in the limit as $s\to1^-$ that
\begin{equation}
	\vol{\opp  p  f}^{-p /n} \geq \vol{\opp  p f^\star}^{- p /n},
\end{equation}
which is equivalent to the P\'olya--Szeg\H o inequality for $L^p$ projection bodies by Cianchi, Lutwak, Yang, and Zhang \cite[Theorem 2.1]{CLYZ2009}. Similarly, by Theorem \ref{thm_limitaM} we obtain from Theorem \ref{thm_aPS+} in the limit as $s\to1^-$ that
\begin{equation}
	\vol{\app  p  f}^{-p /n} \geq \vol{\app  p f^\star}^{- p /n},
\end{equation}
which is equivalent to the P\'olya--Szeg\H o inequality for asymmetric $L^p$ projection bodies by Haberl, Schuster and Xiao  \cite[Theorem 1]{Haberl:Schuster3}.

\section{Affine Fractional Asymmetric $L^p$ Sobolev Inequalities}\label{sec_affine_asymmetric}

We establish the following affine fractional asymmetric $L^p$ Sobolev inequalities and show that they are stronger than Theorem \ref{thm_afracsobo}. 

\begin{theorem}\label{thm_afracsobo+}
    Let $0<s<1$ and $1<p<n/s$. For non-negative $f\in W^{s,p}(\R^n)$, 
    \begin{align}
    \Vert f\Vert_{{\frac {n p}{n-p s}}}^p
    \le
    2\,\sigma_{n,p,s} n\omega_n^{\frac{n+p s}n}\vol{\ppp s p f}^{- \frac{p s}n}
    \le
   2 \sigma_{n,p,s}\int_{\R^n} \int_{\R^n} \frac{\plus{f(x)-f(y)}{p}}{\vert x-y \vert^{n+p s}}\d x\d y.
    \end{align}
    There is equality in the first inequality if and only if $f=h_{s,p}\circ \phi$ for some $\phi\in\gln$ where $h_{s,p}$ is an extremal function of \eqref{eq_fracsob}. There is equality in the second inequality if $f$ is radially symmetric.
\end{theorem}

\begin{proof}
By Theorem \ref{thm_aPS+},
\begin{equation}
	\vol{\ppp s p  f}^{-p s/n} \geq \vol{\ppp s p f^\star}^{- p s/n},
\end{equation}
with equality if $f$ is a translate of  $f^\star\circ \phi$ for some $\phi \in\sln$.	
Since $f^\star$ is radially symmetric, $\ppp s p f^\star=\ppm s p f^\star$ is a ball. Hence, it follows from  \eqref{eq_dualmixed_func+} 
that
\begin{align}
2 n\omega_n^{\frac{n+p s}n} \vol{\ppp s p f^\star}^{-\frac{p s}n} 
&=  2 \int_{\R^n} \int_{\R^n} \frac{\plus{ f^\star(x)-f^\star(y)}p}{\vert x-y \vert^{n+p s}}\d x\d y\\
&=  \int_{\R^n} \int_{\R^n} \frac{\vert f^\star(x)-f^\star(y)\vert^p}{\vert x-y \vert^{n+p s}}\d x\d y.
\end{align}
The fractional Sobolev inequality \eqref{eq_fracsob} shows that
\begin{equation}
    \sigma_{n,p,s}\int_{\R^n} \int_{\R^n} \frac{\vert f^\star(x)-f^\star(y)\vert^p}{\vert x-y \vert^{n+p s}}\d x\d y\ge \Vert f^\star\Vert_{{\frac {n p}{n-p s}}}^p.
\end{equation}
Combining these inequalities and their equality cases, we complete the proof of the first inequality of the theorem.

For the second inequality, we set $K=\B$ in \eqref{eq_dualmixed_func+} and apply the dual mixed volume inequality \eqref{eq_mixedvolume} to obtain
\begin{equation}
    \int_{\R^n} \int_{\R^n} \frac{\plus{f(x)-f(y)}{p}}{\vert x-y \vert^{n+p s}}\d x\d y = n \tilde V_{-p s}(\B, \ppp s p f)\ge n \omega_n^{\frac{n+p s}{n}}\vol{\ppp s p f}^{-\frac{p s}n}.
\end{equation}
There is equality precisely if $\ppp s p f$ is a ball, which is the case for radially symmetric functions.
\end{proof}

Note that it follows from the definition of fractional symmetric and asymmetric $L^p$ polar projection bodies that 
\begin{equation}   
    \fpp s p f =\ppp s p f \ra{-p s} \ppm s p f. 
\end{equation}
We use the dual Brunn--Minkowski inequality \eqref{eq_dualBM} and obtain that
\begin{align}
    \vol{\fpp s p  f}^{-\frac{p s}n} 
    \ge  \vol{\ppp s p f }^{-\frac{p s}n} +\vol{ \ppm s p f}^{-\frac{p s}n},
\end{align}
with equality precisely if the star bodies $\ppp s p f$ and $\ppm s p f$ are dilates. Thus, it follows that for non-negative $f$,  Theorem~\ref{thm_afracsobo+} implies Theorem \ref{thm_afracsobo}  and  it is, in general, substantially stronger than Theorem \ref{thm_afracsobo}. Of course, they coincide for even functions.

\section{Affine Fractional $L^p$ Sobolev Inequalities: Proof of Theorem \ref{thm_afracsobo}}\label{sec_affine_symmetric}

For non-negative $f$, the first inequality in Theorem \ref{thm_afracsobo} follows from Theorem \ref{thm_afracsobo+}, as mentioned before.
For general $f$ and $x,y\in\R^n$, we use 
    \[\vert f(x) - f(y) \vert \geq \big\vert \vert f(x) \vert - \vert f(y) \vert \big\vert,\] 
where equality holds if and only if $f(x)$ and $f(y)$ are both non-negative or non-positive. We obtain
    \[
    \int_{\R^n}\int_{\R^n} \frac{\vert f(x)-f(y)\vert^p}{\vert x-y\vert^{n+s p}}\d x\d y \geq \int_{\R^n}\int_{\R^n} \frac{\big\vert \vert f(x)\vert -\vert f(y)\vert \big\vert^p}{\vert x-y\vert^{n+s p}}\d x\d y,
    \]
with equality if and only if $f$ has constant sign for almost every $x,y \in \R^n$. Using the result for $\vert f\vert$, we obtain  the first inequality of the theorem and its equality case.

For the second inequality, we set $K=\B$ in \eqref{eq_dualmixed_func} and apply the dual mixed volume inequality \eqref{eq_mixedvolume} as in the proof of Theorem \ref{thm_afracsobo+}.

\goodbreak
\section{Optimal Fractional $L^p$ Sobolev Bodies}\label{sec_optimal}

The following important question was asked by Lutwak, Yang and Zhang \cite{LYZ2006} for a given $f\in W^{1,p}(\R^n)$ and $1\le p< n$:
For which origin-symmetric convex bodies $K\subset \R^n$ is
\begin{equation}\label{eq_optimalsobolev}
    \inf\Big\{\int_{\R^n}\|\nabla f(x) \|^p_{K^*}\d x: K \text{ origin-symmetric convex body}, \vol{K}=\omega_n\Big\} 
\end{equation}
attained? An optimal $L^p$ Sobolev body of $f$ is a convex body where the infimum is attained.

Lutwak, Yang ang Zhang \cite{LYZ2006} showed that the infimum in \eqref{eq_optimalsobolev} is attained (up to normalization) at the unique origin-symmetric convex body $\os p f$ in $\R^n$ such that 
\begin{equation}\label{eq_LYZ}
    \int_{\sn} g(\xi )\d S_p(\os p f,\xi)=\int_{\R^n} g(\nabla f (x))\d x
\end{equation}
for every even $g\in C(\R^n)$ that is positively homogeneous of degree $p$,
where $S_p(K,\cdot)$ is the $L_p$ surface area measure of  $K$. 
Setting $g=\Vert\cdot\Vert_{K^*}$, they obtain from the $L^p$~Minkowski inequality that
\begin{equation}\label{eq_funMink}
    \frac 1 n\int_{\R^n}\|\nabla f(x) \|^p_{K^*}\d x= V_p(\os p f, K)\ge \vol{\!\os p f\!}^{(n-p)/n} \vol{K}^{p/n},
\end{equation}
with equality precisely if $K$ and $\os p f$ are homothetic (see \cite[Section 9.1]{Schneider:CB2} for the definition of the $L_p$ mixed volume $V_p(\cdot,\cdot)$ and the $L^p$ Minkowski inequality). 
Hence, they obtain from their solution to their functional version \eqref{eq_LYZ} of the $L^p$ Minkowski problem that $\os p f$ is the optimal $L^p$ Sobolev body associated to $f$. Tuo Wang \cite{Tuo_Wang} obtained corresponding results for $f\in BV(\R^n)$ and $p=1$.

\goodbreak
Let $0<s<1$ and $1<p<n/s$. The results by Lutwak, Yang and Zhang \cite{LYZ2006} suggest the following question for  a given $f\in W^{s,p}(\R^n)$: For which star bodies $L\subset\R^n$  is
\begin{equation}\label{eq_optimalfracsobo}
    \inf\Big\{\int_{\R^n} \int_{\R^n} \frac{\vert f(x)-f(y)\vert^p}{\Vert x-y\Vert_L^{n+p s}}\d x\d y: L \text{ star body}, \vol{L}=\omega_n \Big\}
\end{equation}
attained? An optimal $s$-fractional $L^p$ Sobolev body of $f$ is a star body where the infimum is attained.

By \eqref{eq_dualmixed_func} and the dual mixed volume inequality \eqref{eq_mixedvolume},
\begin{equation}
    \frac 1 n\int_{\R^n}\int_{\R^n} \frac{\vert f(x)-f(y)\vert^p}{\Vert x-y\Vert_L^{n+p s}}\d x\d y =\tilde V_{-p s}(L, \fpp{s} p f)\ge \vol{L}^{(n+p s)/n} \vol{\fpp {s} p f}^{-(p s)/n},
\end{equation}
and there is equality precisely if $L$ is a dilate of $\fpp{s}p f$. Hence, $\fpp s p f$ is the unique optimal $s$-fractional $L^p$ Sobolev body associated to $f$.  

\goodbreak
To understand how the solutions to \eqref{eq_optimalsobolev} and \eqref{eq_optimalfracsobo} are related, we use the following result:  For $f\in W^{1,p}(\R^n)$ and $L\subset \R^n$ a star body, 
\begin{equation}\label{eq_sobolim1}
    \lim_{s\to 1^-}p (1-s)\,\int_{\R^n}\int_{\R^n} \frac{\vert f(x)-f(y)\vert^p}{\Vert{x-y}\Vert^{n+p s}_{L}} \d x\d y
    = \int_{\R^n} \Vert \nabla f(x) \Vert_{\omp L}\d  x,
\end{equation}
where the convex body $\om K$, defined for $\xi\in\sn$ by
\begin{equation}\label{eq_moment}
    h_{\om\! L}(\xi)^p = \int_{\sn}  \vert\langle \xi, \eta\rangle\vert^p \rho_L(\eta)^{n+p} \d \eta,
\end{equation}
is a multiple of the $L^p$ centroid body of $L$. This can be proved as in \cite{Ludwig:fracnorm}, where the corresponding result was established for a convex body $L$ (with a different normalization of $\om L$). It also follows from Theorem \ref{thm_limitM}. Indeed, by \eqref{eq_dualmixed_func} and \eqref{eq_limitK},
\begin{align}
    \lim_{s\to 1^-}p(1-s)\,\int_{\R^n}\int_{\R^n} \frac{\vert f(x)-f(y)\vert^p}{\Vert{x-y}\Vert^{n+p s}_{L}} \d x\d y=  \tilde V_{-p}(L, \opp pf).
\end{align}
Using that
\begin{equation}\label{eq_lyz}
    \opp pf= \opp p \os{p} f
\end{equation}
for $f\in W^{1,p}(\R^n)$, which follows from \eqref{eq_LYZ} by setting $g= \vert\langle \cdot,\eta\rangle\vert^p$ for $\eta\in\sn$ and using \eqref{eq_defop} and \eqref{eq_fpp_def} (cf.~\cite{LYZ2006}),
and that
\begin{equation}\label{eq_interchange}
    V_p(K, \om L)=\tilde V_{-p}(L, \opp p K)
\end{equation}
for $K$ a convex body and $L$ a star body, a well-known relation that follows from Fubini's theorem,  we now obtain \eqref{eq_sobolim1} from the first equation in \eqref{eq_funMink}. 

\goodbreak
Using \eqref{eq_sobolim1}, we  obtain from \eqref{eq_optimalfracsobo} in the limit as $s\to1^-$ for a given $f\in W^{1,p}(\R^n),$ the following question: For which star bodies $L\subset\R^n$ is
\begin{equation}\label{eq_optnewsob}
    \inf\Big\{\int_{\R^n}\|\nabla f(x) \|_{\omp L}\d x: L \text{ star body}, \vol{L}=\omega_n\Big\} 
\end{equation}
attained?
By \eqref{eq_funMink} and the dual mixed volume inequality \eqref{eq_mixedvolume}, we have
\begin{equation}
    \frac 1n\int_{\R^n} \|\nabla f(x) \|^p_{\omp L}\d x =  V_p(\os p {f},\om L)
    = \tilde V_{-p}(L, \opp p f)
\ge \vol{L}^{(n+p)/n} \vol{\opp p f}^{-p/n},
\end{equation}
with equality precisely if $L$ and $\opp pf$ are dilates, where we have used \eqref{eq_lyz} and \eqref{eq_interchange}. From Theorem \ref{thm_limitM}, we obtain that a suitably scaled sequence of optimal $s$-fractional Sobolev bodies converges to a multiple of the optimal body for \eqref{eq_optnewsob} as $s\to 1^-$.

\subsection*{Acknowledgments}
J.~Haddad was partially supported by CNPq, reference no. PQ 305559/2021-4 and by PAIDI 2020, reference no.\ PY20$\_$00664.
M.~Ludwig was supported, in part, by the Austrian Science Fund (FWF):  P~34446.


\begin{thebibliography}{10}

\bibitem{AlmgrenLieb} F.~J. Almgren, Jr. and E.~H. Lieb, {\em Symmetric decreasing rearrangement is  sometimes continuous}, J. Amer. Math. Soc. {\bf 2} (1989), 683--773.

\bibitem{Aubin1976} T.~Aubin, {\em Probl\`emes isop\'erim\'etriques et espaces de {S}obolev}, J.  Differential Geometry {\bf 11} (1976), 573--598.

\bibitem{cafarelli2006} I. Athanasopoulos and L. A. Caffarelli, {\em Optimal regularity of lower-dimensional obstacle problems}, J. Math. Sci. (N.Y.) {\bf 132} (2006), 274--284.

\bibitem{BourgainBrezisMironescu} J.~Bourgain, H.~Brezis, and P.~Mironescu, {\em Another look at {S}obolev spaces}, In: {O}ptimal Control and Partial Differential Equations ({J}. {L}.  {M}enaldi, {E}. {R}ofman and {A}. {S}ulem, eds.). {A} volume in honor of {A}.  {B}ensoussans's 60th birthday, {Amsterdam: IOS Press; Tokyo: Ohmsha}, 2001.

\bibitem{BBM}  J. Bourgain, H. Brezis, and P. Mironescu, {\em Limiting embedding theorems for $W^{s,p}$ when $s \to 1$ and applications}, J. Anal. Math. {\bf 87} (2002), 77--101.

\bibitem{BMS} L. Brasco, S. Mosconi, and M. Squassina, {\em Optimal decay of extremals for the fractional Sobolev inequality}, Calc. Var. Partial Differential Equations {\bf 55} (2016), 1--32.
 
\bibitem{BLL} H.~J. Brascamp, E.~H. Lieb, and J.~M. Luttinger, {\em A general rearrangement inequality for multiple integrals}, J. Funct. Anal. {\bf 17} (1974), 227--237.
 
\bibitem{Burchard96} A.~Burchard, {\em Cases of equality in the {R}iesz rearrangement inequality}, Ann. of Math. (2) {\bf 143} (1996),  499--527.

\bibitem{cafarelli2012} L. Caffarelli, A. Mellet, and  Y. Sire, {\em Traveling waves for a boundary reaction–diffusion equation}, Adv. Math. {\bf 230} (2012), 433--457.

\bibitem{Carlen2017} E.~Carlen, {\em Duality and stability for functional inequalities}, Ann. Fac. Sci. Toulouse Math. (6) {\bf 26} (2017), 319--350.

\bibitem{CLYZ2009} A.~Cianchi, E.~Lutwak, D.~Yang, and G.~Zhang, {\em Affine {M}oser-{T}rudinger and {M}orrey-{S}obolev inequalities}, Calc. Var. Partial Differential  Equations {\bf 36} (2009), 419--436.

\bibitem{FrankSeiringer} R.~Frank and R.~Seiringer, {\em Non-linear ground state representations and sharp {H}ardy inequalities}, J. Funct. Anal. {\bf 255} (2008), 3407--3430.
  
\bibitem{Gardner} R.~Gardner, {\em Geometric {T}omography}, Second ed., Encyclopedia of  Mathematics and its Applications, vol.~58, Cambridge University Press,  Cambridge, 2006.
  
\bibitem{Haberl:Schuster2} C. Haberl and F.E. Schuster, {\em Asymmetric affine $L_p$ Sobolev inequalities}, J. Funct. Anal. {\bf 257} (2009), 641--658.

\bibitem{Haberl:Schuster3} C. Haberl, F.E. Schuster, and J. Xiao, {\em An asymmetric affine Pólya–Szegö principle}, Math. Ann. {\bf 352} (2012), 517–542.

\bibitem{HaddadLudwig_fracsob} J.~Haddad and M.~Ludwig, {\em Affine fractional Sobolev and isoperimetric inequalities}, arXiv:2207.06375 (2022).

\bibitem{andreas} A.~Kreuml, {\em The anisotropic fractional isoperimetric problem with respect to unconditional unit balls}, Commun. Pure Appl. Anal. {\bf 20}  (2021), 783--799.  

\bibitem{Lieb83} E.~Lieb, {\em Sharp constants in the {H}ardy-{L}ittlewood-{S}obolev and related inequalities}, Ann. of Math. (2) {\bf 118} (1983), 349--374.
  
\bibitem{Ludwig:Minkowski} M.~Ludwig, {\em Minkowski valuations}, Trans. Amer. Math. Soc. {\bf 357} (2005), 191--4213.  

\bibitem{Ludwig:fracperi} M.~Ludwig, {\em Anisotropic fractional perimeters}, J. Differential Geom. {\bf 96} (2014), 77--93.

\bibitem{Ludwig:fracnorm} M.~Ludwig, {\em Anisotropic fractional {S}obolev norms}, Adv. Math. {\bf  252} (2014), 150--157.

\bibitem{Lutwak75} E.~Lutwak, {\em Dual mixed volumes}, Pacific J. Math. {\bf 58} (1975), 531--538.

\bibitem{LYZ2000} E.~Lutwak, D.~Yang, and G.~Zhang, {\em ${L}\sb p$ affine isoperimetric  inequalities}, J. Differential Geom. {\bf 56} (2000), 111--132.

\bibitem{LYZ2002b} E.~Lutwak, D.~Yang, and G.~Zhang, {\em Sharp affine ${L}_p$ {S}obolev   inequalities}, J. Differential Geom. {\bf 62} (2002), 17--38.

\bibitem{LYZ2006} E.~Lutwak, D.~Yang, and G.~Zhang, {\em Optimal {S}obolev norms and the {$L\sp  p$} {M}inkowski problem}, Int. Math. Res. Not. (2006), Art. ID 62987, 21.  
  
\bibitem{Ma2014} D.~Ma, {\em Asymmetric anisotropic fractional {S}obolev norms}, Arch. Math. (Basel) {\bf 103} (2014), 167--175.

\bibitem{Mazya} V.~Maz'ya, {\em Sobolev {S}paces with {A}pplications to {E}lliptic {P}artial {D}ifferential {E}quations}, augmented ed., Grundlehren der Mathematischen Wissenschaften, vol. 342, Springer, Heidelberg, 2011.

\bibitem{Schneider:CB2} R.~Schneider, {\em Convex {B}odies: the {B}runn-{M}inkowski {T}heory}, {S}econd expanded ed., Encyclopedia of Mathe\-matics and its Applications, vol. 151, Cambridge University Press, Cambridge, 2014.
  
\bibitem{silvestre2007} L. Silvestre, {\em Regularity of the obstacle problem for a fractional power of the Laplace operator}, Comm. Pure Appl. Math. {\bf 60} (2007), 67--112.

\bibitem{Talenti1976} G.~Talenti, {\em Best constant in {S}obolev inequality}, Ann. Mat. Pura Appl. {\bf 110} (1976), 353--372.

\bibitem{Tuo_Wang} T.~Wang, {\em The affine {S}obolev-{Z}hang inequality on {BV}$({\R}^n)$}, Adv.  Math. {\bf 230} (2012), 2457--2473.
  
\bibitem{Zhang99} G.~Zhang, {\em The affine {S}obolev inequality}, J. Differential Geom. {\bf 53}   (1999), 183--202.

\end{thebibliography}
\end{document}